\documentclass[11pt,reqno]{amsart}
\usepackage{amsmath,amssymb,amsthm,amsfonts,enumerate,hyperref}
\usepackage{graphicx}

\theoremstyle{plain}
\newtheorem{theorem}{Theorem}[section]
\newtheorem{proposition}[theorem]{Proposition}

\theoremstyle{definition}

\theoremstyle{remark}

\numberwithin{equation}{section}

\newcommand{\Ric}{\mathrm{Ric}}

\begin{document}
\title{Bounding $\lambda_{2}$ for K\"ahler-Einstein metrics with large symmetry groups}
\begin{abstract}
We calculate an upper bound for the second non-zero eigenvalue of the scalar Laplacian, $\lambda_{2}$, for toric K\"ahler-Einstein metrics in terms of the polytope data. We also give a similar upper bound for Koiso-Sakane type K\"ahler-Einstein metrics. We provide some detailed examples in complex dimensions 1, 2 and 3.
\end{abstract}
\author{Stuart J. Hall}

\address{Department of Applied Computing, University of Buckingham, Hunter St., Buckingham, MK18 1G, U.K.} 
\email{stuart.hall@buckingham.ac.uk}

\author{Thomas Murphy}

\address{Department of Mathematics, McMaster University, 1280 Main St. W., Hamilton ON, L8S 4M7, Canada}

\email{tmurphy@math.mcmaster.ca}

\maketitle
\section{Introduction}
The purpose of this article is to investigate the second non-zero eigenvalue (ignoring multiplicities) of the scalar Laplacian for a large class of K\"ahler-Einstein manifolds. There are many well-known results concerning the first non-zero eigenvalue of a Riemannian manifold, yet relatively little in the literature about the second eigenvalue. This is partly because unless the Riemannian metric has a very simple form, it is very difficult to calculate exact or even approximate eigenvalues.

The first class of manifolds we study in this article are known as toric-K\"ahler-Einstein manifolds.  They are in some sense the largest group of Fano K\"ahler-Einstein metrics that are currently known, thanks to an existence result of Wang  and Zhu \cite{WZ} (c.f. Theorem \ref{WZthm}). The first non-zero eigenvalue $\lambda_{1}$ on these manifolds is always $2\Lambda$, where $\Lambda$ is the Einstein constant defined by ${\Ric(g) =\Lambda g}$, hence $\lambda_{2}$ is the first geometrically interesting value of the spectrum in this case. The second class of manifolds we consider are Fano manifolds that admit a dense one-parameter family of hypersurfaces (each hypersurface is a principal $U(1)$-bundle over the product of some Fano K\"ahler-Einstein manifolds). Many of these manifolds admit a K\"ahler-Einstein metric due to an existence result of Sakane and Koiso \cite{Sak}, \cite{KoiSak}. As in the Wang-Zhu case, the first eigenvalue of these metrics is $2\Lambda$.

Unfortunately the metrics given by the Wang-Zhu theorem are not known explicitly in many cases and so calculating quantities of interest to physicists and mathematicians is impossible. In the last five years or so there have been huge advances in the use of numerical methods to approximate these metrics and calculate geometric quantities (\cite{Braun}, \cite{Donum}, \cite{DoranC}, \cite{HW}, \cite{Kel}).  For example on the complex surface $\mathbb{CP}^{2}\sharp 3 \overline{\mathbb{CP}}^{2}$ C. Doran et. al. obtained estimates for the first couple of $D_{6}$-invariant eigenvalues and eigenfunctions. However it is not clear that the numerical methods (which will work in theory) can be practically applied in higher-dimensions using current computational power.

We give a method for generating an estimate for $\lambda_{2}$ that could be calculated for any toric K\"ahler-Einstein metric without any explicit representation or approximation of the metric. The main theorem we prove is the following:
\begin{theorem}\label{main}
Let $(M,g,J)$ be an $n$-complex dimensional toric K\"ahler-Einstein manifold with Einstein constant $\Lambda$ and moment polytope $P\subset \mathbb{R}^{n}$. Let $E=\{x_{1},...,x_{n}\}$ be a set of coordinates on the polytope chosen so that
\begin{enumerate}
\item $$\int_{P} x_{i} \ dx =0,$$ for each $i=1,...n$.
\item $$\int_{P} x_{i}x_{j} \ dx = \delta_{ij},$$ for each $i,j = 1,...,n$.
\end{enumerate}
Given $(a_{1},...,a_{n}) \in \mathbb{R}^{n}/\{0\}$ we can form the quadratic function $${\phi^{2}=\left(\sum_{i=1}^{n}a_{i}x_{i}\right)^{2}}.$$ 
Let $\Phi$ is the projection of $\phi^{2}$ onto the $L^{2}$-orthonormal complement of $E\oplus \mathbb{R}$. Then the second eigenvalue of the Laplacian (counted without multiplicity), $\lambda_{2}$, satisfies the bound
$$\lambda_{2} \leq \frac{8\Lambda}{3}+\frac{2\Lambda}{3}\inf_{(a_{1},...,a_{n})\in \mathbb{R}^{n}/\{0\}}\left(\frac{\sum_{i=1}^{i=n}\langle x_{i},\phi^{2} \rangle^{2}_{L^{2}}+4\frac{\|\phi\|^{4}_{L^{2}}}{Vol(P)}}{\|\Phi\|^{2}_{L^{2}}}\right).$$
\end{theorem}
Though the bound is rather unwieldy, the main utility is that the integrals appearing in it can all be easily calculated.  We do this for the toric Fano Kahler-Einstein surfaces and for the threefold  $\mathbb{P}(\mathcal{O}\oplus \mathcal{O}(-1,1))_{\mathbb{CP}^{1}\times \mathbb{CP}^{1}}$.  We also give an actual value (where this is known) or an estimate in the case of $\mathbb{CP}^{2}\sharp3 \overline{\mathbb{CP}}^{2}$ where the metric can only be approximated. We also give a numerical estimate for the threefold  $\mathbb{P}(\mathcal{O}\oplus \mathcal{O}(-1,1))_{\mathbb{CP}^{1}\times \mathbb{CP}^{1}}$ where the K\"ahler-Einstein metric can be explicitly described but the eigenvalues of the Laplacian are not analytically known. The more precise values are of interest in their own right as they do not appear to have been computed in the literature already.  The numerical results are best summarised in the following table (here each of the metrics is normalised so that the Einstein constant $\Lambda=1$):
\begin{center}
\begin{table}[h]
\begin{tabular}{|c|c|c|} 
\hline
\textbf{Manifold} & \textbf{Upper Bound} & \textbf{Actual Value} \\
\hline
$\mathbb{CP}^{1}$ & 6 & 6\\
\hline
$\mathbb{CP}^{2}$ & $16/3$ & $16/3$\\
\hline
$\mathbb{CP}^{1} \times \mathbb{CP}^{1} $ & $ 32/7 $ & $4$\\ 
\hline
$\mathbb{CP}^{2}\sharp 3 \overline{\mathbb{CP}}^{2}$ & $672/127 \approx 5.29$ & $\approx 4.75$\\  
\hline
$\mathbb{P}(\mathcal{O}\oplus \mathcal{O}(1,-1))$ & $\approx 4.70$ & $\approx 4.34$\\
\hline
\end{tabular}
\caption{Estimates and exact values of $\lambda_{2}$}
\end{table}
\end{center}
Here the actual value also refers to the actual value of the second torus invariant non-zero eigenvalue.  This might be strictly greater than the value over all functions, see for example \cite{AbrFr}.  It would be interesting to know if this is actually the case.

The second main result is the following (all notation is explained in section 5): 
\begin{theorem}\label{main2}
Let $W_{q_{1},q_{2},...,q_{r}}$ be a Koiso-Sakane K\"ahler-Einstein manifold determined by the tuples $(p_{1},p_{2},...,p_{r})$, $(n_{1},n_{2},...,n_{r})$ and Einstein constant $\Lambda$. Then the second non-zero eigenvalue $\lambda_{2}$ (counted without multiplicity) satisfies
$$\lambda_{2} \leq \frac{8\Lambda}{3}+ \frac{2\Lambda}{3}\left(\frac{ \frac{I_{3}^{2}}{I_{2}^{2}}+\frac{4}{I_{0}}}{\frac{I_{4}}{I_{2}^{2}}-\frac{I_{3}^{2}}{I_{2}^{2}}-\frac{1}{I_{0}}}\right),$$
where
$$I_{k} = \int_{-(n_{1}+1)}^{(n_{r}+1)}x^{k}\prod_{i=1}^{i=r}\left|\frac{p_{i}}{q_{i}}-x\right|^{n_{i}}dx.$$
\end{theorem}

The results provide the foundation for further investigation into the spectral gap (i.e. the ratio) of the first two non-zero eigenvalues. In particular there is evidence for the following: 

\emph{Conjecture:} Let $(M^{n},g,J)$ be a toric K\"ahler-Einstein manifold with Einstein constant $1$. Then $${\lambda_{2}(M) \leq \lambda_{2}(\mathbb{CP}^{n})},$$
where $\lambda_{2}(\mathbb{CP}^{n})$ denotes the second eigenvalue of the Fubini-Study metric normalised to have Einstein constant 1.

The remainder of the paper is as follows. In section 2 we give some background on toric-K\"ahler manifolds. In section 3 we prove Theorem \ref{main}. In section 4 we calculate the specific examples in Table 1. Finally in section 5 we speculate on some theoretical developments in this area and possible applications to other geometries.

\emph{Acknowledgements} We would like to thank Emily Dryden for useful comments and suggestions.  SH would like to thank Karl Sternberg and Richard Jarman for their hospitality whilst much of this paper was written. We would like to thank the anonymous referee for useful comments and corrections. This work was supported by a Dennison research grant from the University of Buckingham. TM  was supported by an ARC grant. 
 
\section{Toric-K\"ahler manifolds}
Here we give a brief overview of toric-K\"ahler manifolds. The approach we describe was developed by V. Guillemin \cite{Gui} and  M. Abreu \cite{Abr}.  A very good account is also contained in the paper by C. Doran et al. \cite{DoranC}.
\subsection{Background theory}
The class of manifold we are interested in are known as toric-K\"ahler manifolds.  A K\"ahler structure on a manifold is a complex structure $J$ and a $J$-invariant metric $g$ with the condition that the associate $2$-form $\omega=g(J\cdot,\cdot)$ is closed. It is a remarkable fact that, given this data, the metric $g$ can be locally represented in complex coordinates by
$$g_{i\bar{j}}=\frac{\partial ^{2}f}{\partial z_{i}\partial \bar{z}_{j}}$$
where $f$ is a real-valued function.

A toric-K\"ahler manifold $M$ is defined as an $n$-complex dimensional K\"ahler manifold $(M^{n},J,g)$ with an open dense subset $M^{\circ}\subset M$ on which the real $n$-torus $\mathbb{T}^{n}=U(1)^{n}$ acts freely and holomorphically. There are natural complex coordinates on $M^{\circ}$ in this situation, $z=u+\sqrt{-1}\theta$ where $u\in\mathbb{R}^{n}$ $\theta\in\mathbb{T}^{n}$. The torus action then rotates the $\theta$ component and leaves the $u$ fixed. In these coordinates the metric has the form
$$g_{i\bar{j}}=F_{ij}(du_{i}du_{j}+d\theta_{i}d\theta_{j})$$
where $F:\mathbb{R}^{n}\rightarrow \mathbb{R}$ is a convex function. The convex function $F$ induces a change of coordinates given by
$$x=\frac{\partial F}{\partial u}.$$
The map $u\rightarrow \nabla F$ is known as the moment map and the image $P^{\circ}$ of $\mathbb{R}^{n}$ under this map is the interior of a convex polytope $P$.  This polytope $P$ is known as the \emph{moment polytope} of the toric-K\"ahler manifold $M$ and it can be described as the intersection of linear inequalities of the form
$$l_{k}(x)=v_{k}\cdot x+c_{k}\geq 0,$$
where $v_{k}$ is a primitive element of $\mathbb{Z}^{n}$ and $c_{k}\in \mathbb{R}$. The coordinates on the polytope are known as \emph{symplectic coordinates}.  The polytopes that arise in this context are called Delzant polytopes and satisfy the condition that there are $n$ codimension-one faces meeting at each vertex and that the $n$ vectors $v_{k}$ form a basis for $\mathbb{R}^{n}$.

The K\"ahler potential $F$ transforms under the Legendre transform to a function called the \emph{symplectic potential}
$$\psi(x) = u \cdot x - F(u).$$
If we denote by $\psi_{ij} = \frac{\partial^{2}\psi}{\partial x_{i} \partial x_{j}}$ then the metric in symplectic coordinates is then given by
$$g_{ij} =\psi_{ij}dx_{i}dx_{j} +\psi^{ij}d\theta_{i}d\theta_{j},$$
where $\psi^{ij}$ is the inverse of $\psi_{ij}$. Guillemin \cite{Gui} showed that one can always write the symplectic potential as
$$\psi(x) = \sum_{k}l_{k}(x) \log l_{k}(x) + h(x),$$
where $h(x)$ is a smooth function on $P$. We call the function
$$\psi(x)_{can}=\sum_{k}l_{k}(x) \log l_{k}(x),$$
the \emph{canonical symplectic potential}.
\subsection{Toric-K\"ahler-Einstein Metrics}
We are interested in K\"ahler-Einstein metrics with positive Einstein constant, i.e. K\"ahler metrics solving
$$\Ric(g) =\Lambda g$$
with $\Lambda>0$.  These can only occur if the underlying complex manifold is Fano (i.e. has ample anti-canonical bundle).  There are a number of ways of testing whether a toric-K\"ahler metric is Fano by looking at the polytope. One useful characterisation is that there is an affine change of coordinates in which the constants $c_{k}$ in the defining affine linear functions are all equal. Another useful way of looking at this is that the polytope has a preferred center of mass at the origin. We shall henceforth assume that the symplectic coordinates have been chosen so that this is the case. Given a Fano manifold it is not always true that there exists a K\"ahler-Einstein metric. However on toric-K\"ahler Fano manifolds, the only obstruction is essentially classical.\begin{theorem}[Wang-Zhu, \cite{WZ}]\label{WZthm}
Let $(M,g,J)$ be a Fano toric-K\"ahler manifold. Then $M$ has a K\"ahler-Einstein metric, unique up to automorphisms, if and only if the Futaki-invariant vanishes.
\end{theorem}
We shall not expand on the theory of the Futaki invariant, apart from saying that for a Fano toric-K\"ahler manifold it vanishes if, and only if, the centre of mass of the polytope $P$ and its boundary $\partial P$ coincide. This guarantees one can choose coordinates on the polytope satisfying condition (1) in the hypotheses of Theorem \ref{main}.  The Wang-Zhu theorem is far more general than the versions we have stated above.  In the case where the Futaki invariant does not vanish, they show that the manifold admits a K\"ahler-Ricci soliton, which is a generalisation of a K\"ahler-Einstein metric.

The first equivariant eigenspace of a toric K\"ahler-Einstein manifold is easy to calculate. The following fact is well-known and is a manifestation of the classical Matsushima theorem \cite{Mat}. One can deduce it from Theorem 11.52 (page 330) in \cite{Be}.
\begin{proposition}\label{Mat}
Let $(M,g,J)$ be a toric K\"ahler-Einstein metric with Einstein constant $\Lambda >0$, then the first eigenvalue of the scalar Laplacian is $2\Lambda$ and the torus-equivariant eigenfunctions are affine linear functions in the polytope coordinates.
\end{proposition}
 
If the automorphism group is strictly larger than the maximal torus, the first eigenspace will contain eigenfunctions other than the torus-equivariant ones.

\section{The proof of theorem \ref{main}}
In order to obtain our bound we use the characterisation of the second eigenvalue given by
\begin{equation}\label{evquot}
\lambda_{2} = \inf_{\eta \in (E_{1}\oplus \mathbb{R})^{\perp}}\frac{\|\nabla \eta\|^{2}_{L^{2}}}{\|\eta\|^{2}_{L^{2}}},
\end{equation}
where $E_{1}$ is the first eigenspace. Proposition \ref{Mat} gives an extremely concrete discription of the first equivariant eigenspace of a toric-K\"ahler-Einstein manifold. We can integrate powers of first eigenvalues appearing in the quotient (\ref{evquot}) using the following proposition.
\begin{proposition}\label{IBP}
Let $(M,g)$ be a closed Riemannian manifold and let ${\eta \in C^{\infty}(M)}$ be a smooth function. Then for $p,q$ such that $p+q \neq 1$
$$\int_{M}\langle \nabla \eta^{p}, \nabla \eta^{q}\rangle dV_{g} = \frac{pq}{p+q-1}\int_{M}\eta^{p+q-1}\Delta \eta dV_{g}.$$
Hence if one takes $\eta$ an eigenfunction of the Laplacian with eigenvalue $2\Lambda$ we have
$$\int_{M}\langle \nabla \eta^{p}, \nabla \eta^{q}\rangle dV_{g} = \frac{2\Lambda pq}{p+q-1}\int_{M}\eta^{p+q} dV_{g}.$$ 
\end{proposition}
\begin{proof}
We first note the pointwise formulae
$$\langle \nabla \eta^{p},\nabla \eta^{q}\rangle = pq \eta^{p+q-2}|\nabla \eta|^{2}$$
and
$$\eta^{p}\Delta \eta^{q} = q\eta^{p+q-1}\Delta \eta-q(q-1)\eta^{p+q-2}|\nabla \eta|^{2}.$$
Hence
$$p\eta^{p}\Delta \eta^{q}+(q-1)\langle \nabla \eta^{p},\nabla \eta^{q}\rangle = pq\eta^{p+q-1}\Delta \eta.$$
The result follows from integrating by parts.
\end{proof}
Given this result we can now prove Theorem \ref{main}.
\begin{proof} (of Theorem \ref{main})

If $\{x_{1},...,x_{n}\}$ are the coordinates satisfying conditions (1) and (2) in the hypothesis then this is the same as saying 
$$\{x_{1},...,x_{n},\frac{1}{\sqrt{Vol(P)}}\}$$
is an orthonormal basis of $E_{1}\oplus \mathbb{R} \subset L^{2}(M)$. The projection of the quadratic ${\phi^{2}=\left(\sum_{i=1}^{i=n}a_{i}x_{i}\right)^{2}}$ to the orthogonal complement of this space is given by
$$\Phi = \phi^{2}-\sum_{i=1}^{i=n}\langle x_{i}.\phi^{2}\rangle_{L^{2}}\cdot x_{i}-\frac{\langle 1, \phi^{2} \rangle}{Vol(P)}.$$
Straightforward calculation yields 
\begin{multline*}\|\nabla \Phi\|^{2}_{L^{2}} = \|\nabla \phi^{2}\|^{2}_{L^{2}}-2\sum_{i=1}^{i=n}\langle x_{i},\phi^{2}\rangle_{L^{2}}\langle x_{i},\nabla \phi^{2}\rangle_{L^{2}}\\+\sum_{i,j}\langle x_{i},\phi^{2}\rangle_{L^{2}}\langle x_{j},\phi^{2}\rangle_{L^{2}}\langle \nabla x_{i},\nabla x_{j}\rangle_{L^{2}}.\end{multline*}
One can evaluate the integrals $\langle \nabla x_{i}, \nabla \phi^{2}\rangle_{L^{2}}$ and $\langle \nabla x_{i},\nabla x_{j}\rangle_{L^{2}}$ by integration by parts; hence
$$\|\nabla \Phi\|^{2}_{L^{2}} = \|\nabla \phi^{2}\|^{2}_{L^{2}}-2\Lambda\sum_{i=1}^{i=n}\langle x_{i}, \phi^{2}\rangle_{L^{2}}^{2}.$$
The result follows from using Proposition \ref{IBP} and taking the infimum over all possible choices of $(a_{1},...,a_{n})\neq 0$.
\end{proof}
\section{Examples}
\subsection{$\mathbb{CP}^{1}$}
This is the only Fano manifold in complex dimension 1 and the K\"ahler-Einstein metric is the Fubini-Study metric (a.k.a. the round metric). The moment polytope is the line $[-1,1]$ given by the affine linear inequalities 
$$x+1\geq 0 \textrm{ and } 1-x \geq 0.$$
This corresponds to the case with $\Ric (g)=g$ i.e. $\Lambda=1$. 
Straightforward calculation yields
$$\Phi =\frac{ a^{2}}{2}(3x^{2}-1).$$
Hence
$$\|\Phi\|^{2}_{L^{2}} = \frac{2a^{4}}{5}$$
and
$$\|\nabla \Phi\|^{2}_{L^{2}}=\frac{12a^{4}}{5}.$$
Hence $\lambda_{2}\leq 6$. As the Fubini-Study metric is explicit it is well-known that $\lambda_{2}=6$.  When $a=1$, the function $\Phi$ is the Legendre polynomial $\mathcal{P}_{2}(x)$.
\subsection{$\mathbb{CP}^{2}$ }
The K\"ahler-Einstein metric in this case is the Fubini-Study metric which is known explicitly. The moment polytope is the triangle with vertices at $(-1,-1), (2,-1)$ and $(-1,2)$. This corresponds to the normalisation ${\Lambda=1}$. The affine linear functions defining the polytope are
$$x_{1}+1\geq 0,\ \ \ x_{2}+1\geq 0 \textrm{ and } 1-x_{1}-x_{2} \geq 0.$$ 
We form  the orthonormal system of eigenfunctions
$$\tilde{x}_{1} = \frac{2}{3}x_{1}, \ \ \tilde{x}_{2} =\frac{4}{3\sqrt{3}}\left(x_{2}+\frac{x_{1}}{2}\right). $$
The projection of the quadratic $(a\tilde{x}_{1}+b\tilde{x}_{2})^{2}$ is given by
$$\Phi = (a\tilde{x}_{1}+b\tilde{x}_{2})^{2}-\frac{8(a^{2}-b^{2})}{30}\tilde{x}_{1}+\frac{8ab}{15}\tilde{x}_{2}-\frac{2(a^{2}+b^{2})}{9}.$$
We can then calculate
$$\|\Phi\|^{2}_{L^{2}}= \frac{6}{25}(a^{2}+b^{2})^{2}$$
and
$$\|\nabla \Phi\|^{2}_{L^{2}} = \frac{32}{25}(a^{2}+b^{2})^{2}.$$
Hence $\lambda_{2} \leq \frac{16}{3}$.  As with $\mathbb{CP}^{1}$, this is the exact value of $\lambda_{2}$ in this case and $\Phi$ is the associated eigenfunction.
\subsection{$\mathbb{CP}^{1}\times\mathbb{CP}^{1}$}
The K\"ahler-Einstein metric in this case is the product of the Fubini-Study metric on each factor. The moment polytope is the square with vertices at $(-1,-1), (1,-1), (-1,1)$ and $(1,1)$. The affine linear functions defining the polytope are
$$x_{1}+1\geq 0,\ \ \ x_{2}+1\geq 0, \ \ \ 1-x_{1}\geq 0 \textrm{ and } 1-x_{2} \geq 0.$$ 
We form  the orthonormal system of eigenfunctions
$$\tilde{x}_{1} = \frac{\sqrt{3}}{2}x_{1}, \ \ \ \tilde{x}_{2}=\frac{\sqrt{3}}{2}x_{2}.$$ 
The projection of the quadratic is given by
$$\Phi = (a\tilde{x}_{1}+b\tilde{x}_{2})^{2}-\frac{1}{4}(a^{2}+b^{2}).$$
Calculation yields
$$\|\Phi\|^{2}_{L^{2}} =\frac{9}{20}(a^{4}+\frac{10}{3}a^{2}b^{2}+b^{4})-\frac{1}{4}(a^{2}+b^{2})^{2}=\frac{1}{5}(a^{4}+5a^{2}b^{2}+b^{4}). $$
and
$$\|\nabla \Phi\|^{2}_{L^{2}} = \frac{8}{3}(\frac{9}{20}(a^{4}+\frac{10}{3}a^{2}b^{2}+b^{4})).$$
So the ratio is given by
$$\frac{\|\nabla \Phi \|^{2}_{L^{2}}}{\|\Phi\|^{2}_{L^{2}}}=\frac{6a^{4}+20a^{2}b^{2}+6b^{4}}{a^{4}+5a^{2}b^{2}+b^{4}}.$$
This function has a global minima of $\frac{32}{7} \approx 4.6 $ when $a=b$.  The true value of $\lambda_{2}$ for this manifold is $4$.  The eigenfunction in this case is $x_{1}x_{2}$.
\subsection{ $\mathbb{CP}^{2}\sharp3\overline{\mathbb{CP}}^{2}$}
This is the first non-trivial application of the method.  The K\"ahler-Einstein metric on this manifold was first shown to exist by Siu \cite{Siu} but it is not known explicitly. The moment polytope is a hexagon with vertices at $(1,0), (0,1), (-1,1), (-1,0), (0,-1)$ and $(1,-1)$. This corresponds the normalisation $\Lambda=1$. The affine linear functions defining the hexagon are 
$$ x_{i}+1 \geq 0, \ \ \ 1-x_{i}\geq 0 \textrm{ for } i=1,2, $$ 
and
$$1-x_{1}-x_{2}\geq 0 \textrm{ and } 1+x_{1}+x_{2}\geq 0.$$
The orthonormal representatives of the first eigenspace are
$$ \tilde{x}_{1}=\sqrt{\frac{6}{5}}x_{1} \textrm{ and } \tilde{x}_{2} = \sqrt{\frac{8}{5}}\left(x_{2}+\frac{1}{2}x_{1}\right).$$
The projection of the quadratic is given by
$$\Phi = (a\tilde{x}_{1}+b\tilde{x}_{2})^{2}-\frac{1}{3}(a^{2}+b^{2}).$$
Calculation yields 
$$\|\Phi\|^{2} = \frac{127}{375}(a^{2}+b^{2})^{2}$$
and 
$$\|\nabla \Phi\|^{2} = \frac{672}{375}(a^{2}+b^{2})^{2}.$$
Hence we obtain an estimate $\lambda_{2} \leq \frac{672}{127} \approx 5.29$. 

In order to check the accuracy of this bound we can use an approximation of the metric in the Rayleigh-Ritz method.  Briefly, the Rayleigh-Ritz approximation method involves taking a set of test functions $S=\{\psi_{1},...,\psi_{N}\}$ and forming the two matrices
$$A_{ij}  =\langle \nabla \psi_{i}, \nabla \psi_{j}\rangle_{L^{2}} \textrm{ and } B_{ij} = \langle \psi_{i}, \psi_{j}\rangle_{L^{2}}.$$
Providing the test functions from a complete spanning set as $N\rightarrow \infty$, the spectrum  of $M=B^{-1}A$ converges to that of the Laplacian.

The metric is invariant under an action of $D_{6}$ and so it makes sense to try and expand the symplectic potential as a series of $D_{6}$-invariant polynomials. Doran et. al. \cite{DoranC} give an approximation of the Siu metric in terms of the functions $U$ and $V$ where
$$U=x_{1}^{2}+x_{1}x_{2}+x_{2}^{2} \textrm{ and } V= x_{1}^{2}x_{2}^{2}(x_{1}+x_{2})^{2}.$$
They use the symplectic potential given by
$$g = g_{can}-0.22412U-0.01450U^{2}-0.00521U^{3}+0.00734V,$$
this yields a metric that satisfies the Einstein condition pointwise to better than $10\%$ (and the global rms error is $0.007$). Using this approximation to the metric and the set 
$$S=\{1, x_{1}, x_{2}, x_{1}^{2}, x_{2}^{2}, x_{1}x_{2}\},$$
the Rayleigh-Ritz approximation is to $4$ d.p.
$$0,1.9986, 2.0003, 4.7548, 4.7625, 6.3288.$$
This suggests that there is an eigenvalue around $4.75$ of multiplicity $2$.  We note that Doran et. al. gave $6.32$ as being approximately the first eigenvalue of the $D_{6}$-invariant functions. 

\subsection{Complex dimension 3}
In complex dimension 3 there are only two toric K\"ahler-Einstein Fano manifolds that do not come from the product of $\mathbb{CP}^{1}$ and a Fano surface.  They are $\mathbb{CP}^{3}$ and the projectivisation of the rank two bundle $\mathcal{O} \oplus \mathcal{O}(1,-1)$ over $\mathbb{CP}^{1}\times \mathbb{CP}^{1}$. The value of $\lambda_{2}$ for $\mathbb{CP}^{3}$ is well known to be $5$ and our method will achieve this as there is a second eigenfunction that is a quadratic in a first eigenfunction. 

Now we consider ${\mathbb{P}(\mathcal{O} \oplus \mathcal{O}(1,-1))_{\mathbb{CP}^{1}\times\mathbb{CP}^{1}}}$.  The K\"ahler-Einstein metric on this manifold was shown to exist by Sakane \cite{Sak}, though it was not given explicitly. The manifold admits a cohomogeneity-one action and so it is possible to write down the symplectic potential explicitly. This was done by Dammerman in his D.Phil. thesis \cite{Dam}. This manifold can also be realised as the blow-up of $\mathbb{CP}^{3}$ at two skew lines.  This yields a moment polytope which is a tetrahedron sawn-off at two skew edges. The moment polytope for this manifold is given by the linear inequalities
$$  x_{i}+1 \geq 0, \textrm{ for } i=1,2,3$$
and
$$1-x_{1} \geq 0, \ \ \ 1-x_{1}-x_{2} \geq 0 \textrm{ and } 1+x_{1}-x_{3} \geq 0.$$
The orthonormal basis is:
$$ \tilde{x}_{1}=\sqrt{\frac{15}{34}}x_1, \ \ \ \tilde{x}_{2}=\sqrt{\frac{30}{79}}\left(x_{2}+\frac{1}{2}x_{1}\right), \ \ \ \tilde{x}_{3}=\sqrt{\frac{30}{79}}\left(x_{3}-\frac{1}{2}x_{1}\right).$$
This yields the projection of the quadratic
\begin{multline*}\Phi = (a\tilde{x}_{1}+b\tilde{x}_{2}+c\tilde{x}_{3})^{2}+ 0.1906(b^{2}-c^{2})\tilde{x}_{1}+0.3812ab\tilde{x}_{2}\\-0.3812ac\tilde{x}_{3}-\frac{3}{22}(a^{2}+b^{2}+c^{2}).\end{multline*}
A numerical check yields that the optimal bound appears to be when a=b=c, this gives the result (to 4 d.p.)
$$\|\Phi\|_{L^{2}}^{2} \approx 1.7050 $$
and
$$\|\nabla \Phi\|^{2}_{L^{2}} \approx 5.9676,$$
hence an upper bound $\lambda_{2} \leq 4.7011$.
In his thesis, Dammerman showed that the symplectic potential for this metric is given by
\begin{multline*} \psi_{can} +\frac{1}{2}[(x-2)\log (x-2) +(-x-2)\log(-x-2)\\+(1+\frac{\sqrt{7}}{7}x)\log (1+\frac{\sqrt{7}}{7}x)+(1-\frac{\sqrt{7}}{7}x)\log (1-\frac{\sqrt{7}}{7}x)].\end{multline*}
We use the Rayleigh-Ritz method with the set
$$S=\{1,x_{1}, x_{2}, x_{3}, x_{1}^{2}, x_{1}x_{2}, x_{1}x_{3}, x_{2}^{2}, x_{2}x_{3}, x_{3}^{2}\},$$
this yields the following approximation to the spectrum (the values are given to 4 d.p.)
$$0, 1.9997, 2.0002, 2.0006, 4.3447, 4.4430, 4.4437, 5.4483, 5.4548, 5.7107.$$

\section{Koiso-Sakane K\"ahler-Einstein Metrics}
Another class of K\"ahler-Einstein manifolds that we can apply our method to are the Koiso-Sakane \cite{KoiSak} manifolds. Here we use the framework outlined by Dancer and Wang \cite{DW}. Let $(V_{i}^{2n_{i}},J_{i},h_{i}), 1\leq i \leq r$ be compact Fano, K\"ahler-Einstein manifolds of real dimension $2n_{i}$ where the first Chern class ${c_{1}(V_{i},J_{i})=p_{i}a_{i}}$, where $p_{i}$ are integers and $a_{i}\in H^{2}(V_{i}, \mathbb{Z})$ are indivisible cohomology classes. The K\"ahler-Einstein metrics are normalised so that $\Ric (h_{i}) = p_{i}h_{i}$. For ${q=(q_{1},...,q_{r})}$ with $q_{i}\in \mathbb{Z}/\{0\}$ and let $P_{q}$ be the principal $U(1)$-bundle over $V:=V_{1}\times ...\times V_{r}$ with Euler class $\sum_{i=1}^{i=r}q_{i}\pi^{\ast}_{i}a_{i}$ and $\pi_{i}$ is the projection from $V$ onto the $i$th factor.\\ 
\\
Let $\theta$ be the principal $U(1)$ connection on $P_{q}$ with curvature $\Omega=\sum_{i=1}^{i=r}  q_{i}\pi^{\ast}_{i}\omega_{i}$ where $\omega_{i}$ is the K\"ahler form of the metric $h_{i}$. For an open interval $I$, we consider metrics of the form
\begin{equation}\label{Coh1}
\alpha^{-1}ds^{2}+\alpha(s)\theta \otimes \theta +\sum_{i=1}^{i=r}\beta_{i}(s)\pi^{\ast}_{i}h_{i},
\end{equation}
on the manifold $I\times P_{q}$ where $\alpha$ and $\beta_{i}$ are smooth, non-negative functions  on $I$. One can compactify this manifold by collapsing a circle (or a higher-dimensional sphere) at each end. We refer to this manifold as $W_{q_{1},...,q_{r}}$.  In order for metrics of the form (\ref{Coh1}) to extend smoothly to the compact manifold $\alpha$ and $\beta_{i}$ must satisfy certain boundary behaviour.  We refer the reader to \cite{DW} for details here. The class of manifolds we are interested in in this section is provided by the following:
\pagebreak
\begin{theorem}[Koiso-Sakane]
Suppose $0<|q_{i}|<|p_{i}|$ for all $1\leq i \leq r$ and the integral
\begin{equation}\label{Fut}
\int_{-(n_{1}+1)}^{(n_{r}+1)}x\prod_{i=1}^{i=r}\left(x-\frac{p_{i}}{q_{i}}\right)^{n_{i}}dx=0.
\end{equation}
Then $W_{q_{1},...,q_{r}}$ admits a K\"ahler-Einstein metric.
\end{theorem}
The integral is the Futaki invariant of the holomorphic vector field $\partial_{s}$, hence, as in the toric case, the only obstruction is classical. In this case the functions $\alpha$ and $\beta_{i}$ are given by
$$\alpha =\frac{1}{\prod_{i=1}^{i=r}\left( s+\sigma_{i}\right)^{n_{i}}}\int_{0}^{s}(2n_{1}+2-x) \prod_{i=1}^{i=r}\left(x+\sigma_{i}\right)^{n_{i}}dx,$$
and,
$$\beta_{i}(s)=-q_{i}\left(s+\sigma_{i}\right),$$
where
$$ \sigma_{i} = -2n_{1}-2-\frac{2p_{i}}{q_{i}}.$$
In the case that the factors in the base are homogeneous the manifolds are also toric, however non-toric examples do exist. The $U(1)$ action on the bundle $P_{q}$ lifts to a holomorphic $U(1)$ action on $W_{q_{1},...q_{r}}$ and so by the Lichnerowicz-Matsushima theorem, these manifolds all have $\lambda_{1}=2\Lambda$.  Furthermore, the potential function for the holomorphic vector field induced by the $U(1)$ action is just the coordinate $s$.  Equation (\ref{Fut}) is essentially equivalent to the requirement that $s-2n_{1}-2$ has integral $0$ and is the first non-constant eigenfunction for the Laplacian.
\begin{theorem}
Let $-(n_{1}+1)q_{i}<p_{i}$ and $(n_{r}+1)q_{i}<p_{i}$ for $2 \leq i \leq r-1$ and suppose that equation (\ref{Fut}) is satisfied, then the second non-zero eigenvalue of the Koiso-Sakane metric satisfies
$$\lambda_{2} \leq \frac{8\Lambda}{3}+ \frac{2\Lambda}{3}\left(\frac{ \frac{I_{3}^{2}}{I_{2}^{2}}+\frac{4}{I_{0}}}{\frac{I_{4}}{I_{2}^{2}}-\frac{I_{3}^{2}}{I_{2}^{2}}-\frac{1}{I_{0}}}\right),$$
where
$$I_{k} = \int_{-(n_{1}+1)}^{(n_{r}+1)}x^{k}\prod_{i=1}^{i=r}\left|\frac{p_{i}}{q_{i}}-x\right|^{n_{i}}dx.$$
\end{theorem}
\begin{proof}
This follows as before by taking the projection of $s^{2}$ to the orthogonal complement of $\{1,s\}$ and then using the integration-by-parts formula to evaluate the integrals in the Rayleigh quotient.
\end{proof}
\subsection{Examples}
The 3-fold $\mathbb{P}(\mathcal{O}\oplus \mathcal{O}(1,-1))$ was orginally realised as a Koiso-Sakane manifold. The data in this case are $n_{1}=n_{4}=0$, $n_{2}=n_{3}=1$, $p_{2}=p_{3}=2$ and $q_{2}=-q_{3}=1$. Here the integrals are given by
$$I_{0} = \frac{22}{3}, \ I_{2} = \frac{34}{15}, \ I_{3}=0, \ I_{4} = \frac{46}{35}.$$
Hence we obtain the bound:
$$ \lambda_{2} \leq \frac{2530}{443} \approx 5.711.$$
We note that this bound is not as good as the one obtained in the toric theorem. This is to be expected as the bound involves a function invariant under a strictly larger symmetry group. It is reassuring that this bound occurs as the largest eigenvalue of the Rayleigh-Ritz matrix computed using Dammerman's symplectic potential.

The example above can easily be generalised bytaking the product of any two projective spaces of the same dimension $N$ and forming manifold $W_{q,-q}$ for $0<q<N+1$. The various bounds given by Theorem \ref{main2} are given below:
\begin{center}
\begin{table}[h]
\begin{tabular}{|c|c|c|} 
\hline
\textbf{N} & \textbf{q} & \textbf{Bound for $\lambda_{2}$} \\
\hline
2 & 1 & 5.7526\\
\hline
2 & 2 & 5.1136 \\
\hline
3  & 1 & 5.7924\\ 
\hline
3 & 2 & 5.2549\\  
\hline
3 & 3 & 4.6750\\
\hline
\end{tabular}
\caption{Bounds for $\lambda_{2}$ of Koiso-Sakane manifolds}
\end{table}
\end{center}
\section{Further Directions}
\subsection{Higher order estimates}
One could use any polynomial in a representative of the first eigenspace to generate an estimate, though general considerations would suggest that one could not really improve on the quadratic bound very much by doing this.  Indeed, in the case of $\mathbb{CP}^{n}$ and products of projective spaces there would be no improvement from doing this. For $\mathbb{CP}^{2}\sharp3\overline{\mathbb{CP}}^{2}$ a numerical check using the Rayleigh-Ritz method does not yield a significant improvement for example,  a quintic bound is given by $\lambda_2 \leq 5.287$.
\subsection{Theoretical considerations}
Finding the quadratic (and higher degree) bounds can be phrased in terms of a very general optimisation problem.  Given any compact subset $C$ of $\mathbb{R}^{n}$ with $Volume(C)>0$ then one can form an orthonormal basis out of the coordinate functions. Choosing a representative $\eta$  one can find the projection of $\eta^{2}$ onto the orthogonal complement of the coordinates  (denoted again by $\Phi$) and try to find the representative that minimises the quantity
$$ \frac{\sum_{i=1}^{i=n}\langle x_{i},\eta^{2}\rangle^{2} + 4\frac{\langle 1,\eta^{2}\rangle^{2}}{Vol(C)}}{\|\Phi\|^{2}_{L^{2}}}. $$
For example if $C=\mathbb{D}^{2}$ is the unit disc then
$$\tilde{x}_{1} = \frac{2}{\sqrt{\pi}}x_{1}, \ \ \ \tilde{x}_{2} = \frac{2}{\sqrt{\pi}}x_{2},$$
are orthonormal and the projection of the quadratic ${(a\tilde{x}_1+b\tilde{x}_{2})^{2}}$ onto the orthogonal complement is given by
$$\Phi = (a\tilde{x}_{1}+b\tilde{x}_{2})^{2}-\frac{1}{\pi}(a^{2}+b^{2}).$$
Hence
$$\|\Phi\|^{2}_{L^{2}} = \frac{1}{\pi}(a^{2}+b^{2})^{2}. $$
If we imagined that the disc was the moment polytope of a toric K\"ahler-Einstein manifold this would yield an upper bound of $\frac{16\Lambda}{3}$ for $\lambda_{2}$. 

It might be possible to prove that there is a universal upper bound (depending only upon the dimension $n$) to this general problem and so a universal upper bound to $\lambda_{2}$ for any toric K\"ahler-Einstein manifold. If we denote by $\lambda_{2}(\mathbb{CP}^{n})$ the second eigenvalue of the Fubini-Study metric on complex projective space then the numerical work certainly gives evidence for the following:

\emph{Conjecture 1:} Let $(M^{n},g,J)$ be a toric K\"ahler-Einstein manifold. Then $$\lambda_{2}(M) \leq \lambda_{2}(\mathbb{CP}^{n}).$$
In fact, perhaps a stronger result is true.

\emph{Conjecture 2:} Let $(M^{n},g,J)$ be a toric K\"ahler-Einstein manifold. Then the
bound given in Theorem 1.1 is less than or equal to $\lambda_{2}(\mathbb{CP}^{n})$.

There has been interest in to what extent the spectrum of a toric-K\"ahler manifold determines the polytope. We refer the reader to the recent work of Dryden, Guillemin and Sena-Dias \cite{Dry} for results in this direction. In particular, they suggest it might be possible that the equivariant spectrum of toric K\"ahler-Einstein metrics might uniquely determine the polytope and hence the manifold.  Any product metric will have $\lambda_{2}=4$ but as the polytopes for toric K\"ahler-Einstein manifolds are very restricted in the set of Delzant polytopes one possibility might be that a non-product, toric K\"ahler-Einstein manifold has polytope uniquely determined by $\lambda_{2}$. In other words, each non-product, toric K\"ahler-Einstein metric (normalised to have Einstein constant 1) has a unique value $\lambda_{2}$ associated to it.

For the Koiso-Sakane manifolds, one could consider the problem of bounding the expression for the bound in terms of integrals involving an arbitrary measure $\mu$ on $[-1,1]$ with center of mass $0$ (i.e. $\int_{-1}^{1} x d\mu=0$). This problem has no universal bound as can be seen by taking the limit as $\epsilon \rightarrow 0$ of  $\mu = V(\epsilon)e^{(1-x^{2})/\epsilon}dx$, where $V(\epsilon)$ is a constant to make $\mu$ have unit mass. The numerical work seems to suggest that $6 = \lambda_{2}(\mathbb{CP}^{1})$ is a universal upper bound for the manifolds $W_{q,-q}$.

\subsection{Other geometries}
Techniques similar to those described above might also be useful in bounding the first (or low lying) eigenvalues of other geometries that are of interest to mathematicians and physicists. For example, in \cite{HM} similar techniques are used to obtain bounds on the first eigenvalue of the Page metric and Chen-LeBrun-Weber metric, both Einstein metrics on $\mathbb{CP}^{2}\sharp\overline{\mathbb{CP}}^{2}$ and $\mathbb{CP}^{2}\sharp 2\overline{\mathbb{CP}}^{2}$ respectively. All one needs is the explicit knowledge of a function that is the solution to an equation involving the Laplacian and explicit knowledge of the volume form.

\end{document}